\documentclass[11pt]{amsart}
\usepackage[latin1]{inputenc}
\usepackage{amsmath}
\usepackage{amsfonts}
\usepackage{amssymb}
\usepackage{extarrows}
\usepackage[all]{xy}
\usepackage{tikz}
\usepackage{tikz-cd}
\usepackage{mathrsfs}

\usepackage{amsthm}
\theoremstyle{plain}
\newtheorem{thm}{Theorem}[section]

\newtheorem{lemma}[thm]{Lemma}
\newtheorem{rmk}[thm]{Remark}
\newtheorem{prop}[thm]{Proposition}

\theoremstyle{definition}
\newtheorem{example}[thm]{Example}
\newtheorem{defn}[thm]{Definition}
\author{Yubo TONG}

\title{The Chevalley-Weil formula on nodal curves}

\date{\today}
\address{Xiamen University  \\ School of Mathematical Sciences \\ Siming South Road 422 \\ Xiamen, Fujian 361005 (China)}
\email{tongyubo@stu.xmu.edu.cn}

\begin{document}

\maketitle
\begin{quote}
\begin{abstract}
 In this paper, we study the eigensubspace of the space of the holomorphic differentials of nodal curves over the algebracally closed field under the action of finite automorphism groups. We compute the Chevalley-Weil formula with some additional contidions of the quotient curve and give some examples.
 \end{abstract}
\end{quote}

\tableofcontents

\section{Introduction}

Let $X$ be a connected projective smooth curve over an algebraically closed field $k$ and $G\subseteq \operatorname{Aut}(X)$ be a finite subgroup. Then $G$ acts in a natural way on the space of the holomorphic differentials on $X$, thus we obtain a linear representation $G \rightarrow \operatorname{GL}\left(H^{0}(X, \omega_{X})\right)$. A basic problem is to determine how many times a given irreducible representation of $G$ occurs in $H^{0}(X, \omega_{X})$.

This problem was first considered by Hurwitz \cite{HZ} for $G$ cyclic over $k=\mathbb{C}$. Then in the $30$s of the 20th century, Chevalley and Weil \cite{CW} solved this problem for general $G$ when $\pi:X\rightarrow X/G$ is unramified. Soon after, Weil \cite{W} solved the case for general $\pi$.
This result was named as the Chevalley-Weil formula and it remains valid for any algebraically closed field $k$ with $\operatorname{char}(k)=p \nmid \#G$ \cite{Kani}.

When $\operatorname{char}(k)=p>0$ and $p\mid \#G$, the structure of $H^{0}(X, \Omega_{X})$ becomes more complicated. Except the tame ramification case
(\cite{Kani}, \cite{N2}), or weakly ramified case(\cite{Kock}), people focus on some special groups (\cite{V} for the case
of cyclic groups, \cite{MW} for abelian groups, \cite{J} for $p$-groups or \cite{BC} for groups with a cyclic Sylow subgroup).

In the 1980s, Kani studied the projectivity of the logarithmic differentials space $H^{0}(X, \Omega_{X}(D))$ as $k[G]-$module in the tamely ramified case\cite{Kani}. But most of his work was covered by Nakajima's work(\cite{N1},\cite{N2}). The latter improved Mumford's method\cite[$\uppercase\expandafter{\romannumeral2} .5$ Lemma 1]{Mumford} to study the $H^{i}\left(X, \mathcal{G}\right)$ of the coherent $G$-sheaf $\mathcal{G}$ in the tamely ramified Galois covering for any dimensional projective varieties. Nevertheless, Kani's work gives us sereval valuable tools.

For smooth curves, the Chevalley-Weil formula was well understood by now.
In this paper, we will follow Kani's methods, and generalize the Chevalley-Weil formula to the nodal curves for one-dimensional $G$-representations with $\operatorname{char}(k)=0$ or prime to $\#G$.\\

\noindent\textbf{Acknowledgements}.
I would like to thank my supervisor Wenfei~Liu for his support.
I am grateful to Professor Qing~Liu for writing advice and helpful discussions during my visit at University of Bordeaux. This work has been supported by the NSFC (No.~11971399) and by the Presidential Research Fund of Xiamen University(No.~20720210006).

\section{Preliminary}
\subsection{Notations}\label{notations}
In this paper, we consider a finite group $G$ acting faithfully on a nodal curve $X$ over an algebraically closed field $k$. Let $\#G=n$ and $\operatorname{char}(k)=p \nmid n$ or $\operatorname{char}(k)=0$, which implies that $k[G]$ is semi-simple. A \emph{curve} means an equidimensional reduced projective scheme of finite type of dimension $1$ over $k$.

Let $X$ be a nodal curve, $\omega_X$ the canonical (dualizing) sheaf of $X$. Let  $\hat{X}\overset{\alpha}{\rightarrow} X$ be the normalizaiton of $X$, then it induces an immersion $H^0(X,\omega_X) \hookrightarrow H^{0}(\hat{X}, \Omega_{\hat{X}}(\hat{S}_X))$, where $\hat{S}_X$ is the preimage of singularities(nodes) of X. For a node $P\in \alpha(\hat{S}_X)$, we say $\{P_1,P_2\}=\alpha^{-1}(P)$ a \emph{pair} of $P$.

An element $\varphi_0\in H^0(X,\Omega_{\hat{X}}(\hat{S}_X))$ belongs to $H^0(X,\omega_X)$ if and only if $\operatorname{Res}_{P_1}\varphi_0+\operatorname{Res}_{P_2}\varphi_0=0$ for any pair $\{P_1,P_2\}$. Such an element is called a holomorphic differential of $X$. It is known that $H^0(X,\omega_X)$ is a $k$-vector space of dimension $p_a$, the arithmetic genus of $X$.

Both the rational function field $K(X)$ and $H^0(X,\omega_X)$ are naturally (right)$k[G]$-modules, and every $1$-dimensional representation is its character. Our goal is to compute the multiplicity of any $1$-dimensional representation $\chi$, that is the dimension of $H^0(X,\omega _X)_{\chi}$ over $k$.
Note that all the irreducible representations will be $1$-dimensional when $G$ is abelian.
\\

Let $X$ be smooth for the rest of this section. Now we recall some properties for smooth curves.\\

Consider the branched cover $\pi:X\rightarrow X/G=Y$ and let $e_P$ be the ramification index at $P\in X$, then
we have the \emph{ramification divisor} 
$$R_{\pi}=\sum_{P\in X}(e_P -1)P.$$
For a divisor $D=\sum a_i P_i \in \operatorname{Div}(X)$, define $\pi_*D \in \operatorname{Div}(Y)$ by
$$
\pi_*D =\sum a_i \pi(P_i).
$$

If $D=\sum a_i Q_i \in \operatorname{Div}(Y)$ is a divisor and $r \in \mathbb{R}$, then define $\lfloor rD \rfloor \in \operatorname{Div}(Y)$ by
	$$
	\lfloor r D \rfloor=\sum \lfloor r a_i \rfloor Q_i,
	$$
	where $\lfloor r a_i \rfloor$ denotes the greatest integer $\leq  r a_i$. And define $\pi^*D \in \operatorname{Div}(Y)$ by
    $$
    \pi^*D =\sum_i a_i(\sum_{P\in \pi^{-1}(Q_i)}  e_P ~ P).
    $$

\begin{prop}[Kani \cite{Kani}]\label{Kani equality} Let $G$ be a finite group acting on a smooth curve $X$ with $R_{\pi}$ the ramification divisor of
	 $\pi: X \rightarrow X/G=Y$. Consider a $G$-invariant divisor $D \in \operatorname{Div}(X)$, then for the trivial character $\chi=1_G$, we have
 \begin{align}\label{G1}
	H^{0}(X, \mathcal{ O}_X(D))^G&=\pi^* H^{0}(Y, \mathcal{O}_Y \left\lfloor n^{-1} \pi_* D\right\rfloor), \\	\label{G2}
	H^{0}(X, \Omega_X(D))^G&=\pi^* 	H^{0}(Y,\Omega_Y \left\lfloor n^{-1} \pi_*(D+R_{\pi}) \right\rfloor).
 \end{align}

For a $1$-dimensional character $\chi$, let $f_\chi \in K(X)^*$ be such that $\sigma f_\chi=\chi(\sigma) f_\chi$  for all $\sigma \in G$ (whose existence is guaranteed by Hilbert's theorem 90). Then
 \begin{align}\label{G3}
	H^{0}(X, \mathcal{O}_X(D))_\chi&=f_\chi \cdot \pi^* H^{0}(Y, \mathcal{O}_Y \left\lfloor n^{-1}\pi_*\left(D+\left(f_\chi\right)\right)\right\rfloor),\\\label{G4}
	H^{0}(X, \Omega_X(D))_\chi&=f_\chi \cdot \pi^* H^0(Y,\Omega_Y\left\lfloor n^{-1} \pi_*\left(D+\left(f_\chi\right)+R_{\pi}\right)\right\rfloor) .
 \end{align}
\end{prop}
\begin{proof}(More details here than in \cite{Kani}.)
Note that $D \geq \pi^*\left\lfloor n^{-1} \pi_* D\right\rfloor$ and hence $H^{0}(X, \mathcal{ O}_X(D))^G \supseteq \pi^* H^{0}(Y, \mathcal{O}_Y \left\lfloor n^{-1} \pi_* D\right\rfloor)$.
 Conversely, if $f \in H^{0}(X, \mathcal{ O}_X(D))^G$, then $f=\pi^* e$ with some $e \in K(Y)$.  Hence $\pi_*((f)+D)=n(e)+\pi_* D \geq 0$, which implies $(e)+\left\lfloor n^{-1} \pi_* D\right\rfloor \geq 0$. This proves (\ref{G1}).
 
To prove (\ref{G2}), fix a meromorphic differential $0 \neq \varphi \in \Omega(Y)$, which always exists by Riemann-Roch. By
$
H^{0}(X, \Omega_X(D))=H^{0}(X, \mathcal{ O}_X \left(D+\left(\pi^* \varphi\right)\right))  \cdot\pi^* \varphi=
H^{0}(X, \mathcal{ O}_X \left(D+\pi^*(\varphi)+R_{\pi}\right)) \cdot \pi^* \varphi,
$
we have
\begin{align}\nonumber
H^{0}(X, \Omega_X(D))^G&=H^{0}(X, \mathcal{ O}_X \left(D+\pi^*(\varphi)+R_{\pi}\right))^G\cdot \pi^*\varphi\\ \nonumber
&=\pi^* H^{0}(Y, \mathcal{O}_Y \left\lfloor n^{-1} \pi_* (\left(D+\pi^*(\varphi)+R_{\pi}\right))\right\rfloor)  \cdot \pi^*\varphi\\ \nonumber
&=\pi^* [H^{0}(Y, \mathcal{O}_Y (\lfloor n^{-1} \pi_* (D+R_{\pi})\rfloor)+(\varphi))  \cdot \varphi]\\ \nonumber
&=\pi^* 	H^{0}(Y,\Omega_Y \left\lfloor n^{-1} \pi_*(D+R_{\pi}) \right\rfloor).
\end{align}
Finally, (\ref{G3}) and (\ref{G4}) for general $\chi$ is followed by 
$$
H^{0}(X, \mathcal{ O}_X(D))_\chi=f_\chi \cdot H^{0}(X, \mathcal{ O}_X\left(D+\left(f_\chi\right)\right))^G,
$$
 $$
 H^{0}(X, \Omega_X(D))_\chi=f_\chi \cdot H^{0} \left( X, \Omega_X (D+(f_\chi) )\right)^G.
 $$
\end{proof}

\subsection{Ramification modules}\cite[Kani]{Kani}
Let $Bl(Y)$ be the \emph{branch locus} of $\pi:X\rightarrow Y$.

Fix a point $P \in X$, and let $G_P$ be the stablizer subgroup of $G$ at $P$, which is a cyclic of order $e_P$. Then there is a unique character
$
\theta_P: G_P \rightarrow k^*
$
such that for any $f\in K(X)^*$,
$$
\frac{\sigma f}{f}\equiv \theta_P(\sigma)^{v_P(f)}(\operatorname{mod} ~\mathfrak{m}_P), \quad \forall \sigma \in G_P,
$$
where $v_P$ denotes the \emph{valuation} at $P$ and $\mathfrak{m}_P$ the maximal ideal of the local ring $\mathcal{ O}_P$. 

Set 
$$
R_{G, P}:=\operatorname{Ind}_{G_P}^G\left(\bigoplus_{d=0}^{e_P -1} d \cdot \theta_P^d\right).
$$
\begin{defn}
For a point $Q\in Y$,
define \emph{the ramification module of $Q$}
$$
R_{G, Q}:=\bigoplus_{P\in \pi^{-1}(Q)}R_{G, P},
$$
and \emph{the ramification module of $\pi$}
 $$R_{G}:=\bigoplus_{Q\in Y}R_{G, Q_i}.$$
Note that this is a finite sum because $R_{G, Q}=0$ for $Q\notin Bl(Y)$.  
\end{defn}

Consider an $f_\chi \in K(X)^*$ such that $\sigma f_\chi=\chi(\sigma) f_\chi$ for any $\sigma \in G$ in Proposition \ref{Kani equality}. Since $f_{\chi}^n \in \pi^* k(Y)$, write $\left(f_{\chi}^n\right)=\pi^*(n A+B)$ where $A, B \in \operatorname{Div}(Y)$ and $\lfloor n^{-1} B\rfloor=0$. 

Note that $\operatorname{Supp}(B)\subseteq Bl(Y)$, so we write $B=\sum_{Q\in Bl} b_Q Q$. By definition, we have
$$
b_Q=n\left\langle\frac{v_{Q}\left(f_{\chi}^n\right)}{n}\right\rangle,
$$
where $\langle r\rangle=r-\lfloor r \rfloor$ denotes the fractional part of $r$. 
The following lemma shows that this $B$ is independent of the choices of $f_{\chi}$.
 
\begin{lemma}\label{bi}
	Let $\chi: G \rightarrow k^*$ be a  $1$-dimensional character. Then for any $Q\in Bl(Y)$, we have
	\begin{equation}\label{fbi}
		n\left\langle\frac{v_{Q}\left(f_{\chi}^n\right)}{n}\right\rangle=\left\langle \chi, R_{G, Q}\right\rangle_G.
	\end{equation}

\end{lemma} 
\begin{proof}
	Let $P \in \pi^{-1}(Q)$. Then by Frobenius reciprocity, we have

	\begin{align}
		\left\langle \chi, R_{G, P}\right\rangle_G  =\left\langle \chi |_{G_P}, \bigoplus_{d=0}^{e_P -1} d \cdot \theta_P^d\right\rangle_{G_P}.
	\end{align}

Note that $\theta_P^d$ runs through are all the irreducible representations of  $G_P$, hence we have
	\begin{equation}
	\left\langle \chi, R_{G, P}\right\rangle_G=a \Leftrightarrow \chi |_{G_P}=\theta_P^a  
	\end{equation}
with $0 \leq a<e_P$.
Choose a generator $\sigma$ of $G_P$, then by definition of $f_\chi$, we have
	$$
	\sigma f_\chi=\chi(\sigma) f_\chi=\theta_P(\sigma)^a f_{\chi}.
	$$
	Furthermore, by the definition of $\theta_P$, we have
	
$$
\theta_P(\sigma)^a=\frac{\sigma f_{\chi}}{f_{\chi}}\equiv \theta_P(\sigma)^{v_P(f_{\chi})}(\operatorname{mod} ~\mathfrak{m}_P),
$$
    which implies $a\equiv v_P(f_{\chi})(\operatorname{mod} ~e_P)$ since $\theta_P(\sigma)$ has order $e_P$ in $k^*$.
     Finally,
    $$
    \frac{\left\langle \chi, R_{G, P}\right\rangle_G}{e_P}=\left\langle \frac{ v_P(f_{\chi})}{e_P}\right\rangle =\left\langle \frac{ v_Q(f_{\chi}^n)}{n}\right\rangle
    =\frac{b_Q}{n},
    $$
    namely we have $\left\langle \chi, R_{G, Q}\right\rangle_G=b_Q$~.

\end{proof}

\section{Irreducible nodal curves}
Let $X$ be an irreducible nodal curve in this section.
\subsection{The $G$-invariant differentials}\label{subsection G}
Let $G$ be a finite group acting on an irreuducible nodal curve $X$ and $Y=X/G$ the quotient curve.
For the space $H^{0}(X, \omega_{X})^G$ of $G$-invariant differentials, it is a classical fact that
\begin{prop}
	If $X$ is smooth, then $$\mathrm{dim}_k H^{0}(X, \Omega_X)^G=g_Y.$$
\end{prop}
\begin{proof}
	With the notations of \S\ref{notations}, let $e_Q:=e_{P}$ for any $P\in \pi^{-1}(Q)$.
	Note that 
	 $$\lfloor n^{-1} \pi_*R_{\pi} \rfloor=
	\sum_{Q\in Y} \left\lfloor \frac{e_Q-1}{e_Q}\right\rfloor Q=0.
	$$
	 By Proposition \ref{Kani equality} (\ref{G2}), we have 
	$$	H^{0}(X, \Omega_X)^G=\pi^* 	H^{0}(Y,\Omega_Y \left\lfloor n^{-1} \pi_*R_{\pi} \right\rfloor)=\pi^* H^{0}(Y,\Omega_Y).$$
\end{proof}

Here comes a natural qustion that for the covering $\pi:X\rightarrow X/G=Y$ of nodal curves, do we still have the equality
\begin{align}\label{G-differentials}
\mathrm{dim}_k H^{0}(X, \omega_X)^G=p_a(Y)?
\end{align}

Consider the normalizations $\hat{X} \rightarrow X$ and $\hat{Y}\rightarrow Y$, respectively. We have $\hat{X}/G=\hat{Y}$, so
there is a commutative diagram
\begin{equation}\nonumber
	\begin{tikzcd}
		\hat{X} \arrow[r, "\hat{\pi}"]\arrow[d] & \hat{Y} \arrow[d]\\
		X\arrow[r, "\pi"] & Y. 
	\end{tikzcd}
\end{equation}

This induces the corresponding morphisms of differentials
\begin{equation}\label{Canonical G-differentials}
	\begin{tikzcd}
		H^{0}(X, \omega_{X}) \arrow [r,dashleftarrow, "?"] \arrow[d,hook] & \pi^*H^{0}(Y, \omega_{Y}) \arrow[d,hook]\\
		H^{0}(\hat{X}, \Omega_{\hat{X}}(\hat{S}_X))\arrow [r,hookleftarrow] & \hat{\pi}^* H^{0}(\hat{Y}, \Omega_{\hat{Y}}(\hat{S}_Y)),
	\end{tikzcd}
\end{equation}
because $X\rightarrow X/G$ takes smooth points to smooth points, so $\hat{\pi}^{-1}(\hat{S}_Y)\subseteq \hat{S}_X $.

\begin{lemma}\label{upper row}
	The upper row 
	$$
    \pi^*H^{0}(Y, \omega_{Y}) \subseteq H^{0}(X, \omega_{X})
	$$
	of (\ref{Canonical G-differentials}) exists if and only if the ramification indexes $e_{P_1}=e_{P_2}$ for all pairs $\{P_1,P_2\}\subseteq \hat{S}_X$.
\end{lemma}
\begin{proof}
    Given some $\varphi_Y\in H^{0}(Y, \omega_{Y})$, we have $\operatorname{Res}_{\hat{\pi}(P_1)}\varphi_Y=\operatorname{Res}_{\hat{\pi}(P_2)}\varphi_Y$.
    Note that for any $P\in \hat{X}$, we have $\operatorname{Res}_{P}(\hat{\pi}^*\varphi_Y)=e_P\cdot \operatorname{Res}_{\hat{\pi}(P)}\varphi_Y$ .
    Hence for any pair $\{P_1,P_2\}\subseteq \hat{S}_X$, we have
    \begin{equation}
      \operatorname{Res}_{P_1}(\hat{\pi}^*\varphi_Y)=\operatorname{Res}_{P_2}(\hat{\pi}^*\varphi_Y)
    \end{equation} 
    if and only if $e_{P_1}=e_{P_2}$.
\end{proof}

Note that the points of $\hat{S}_X-\hat{\pi}^{-1}(\hat{S}_Y)$ are mapped to the smooth part of $Y$. 
\begin{lemma}\label{trivial}
	For the left column of (\ref{Canonical G-differentials}), we have
	$$H^{0}(X, \omega_{X})^G \hookrightarrow H^{0}(\hat{X}, \Omega_{\hat{X}}( \pi^{-1}(\hat{S}_Y)))^G.$$
\end{lemma}

\begin{proof}
	Let $\varphi\in H^{0}(X, \omega_{X})^G$, and a pair $\{P_1,P_2\}\subseteq \hat{S}_X-\hat{\pi}^{-1}(\hat{S}_Y)$. Then there exists some $\sigma\in G$ such that $\sigma(P_1)=P_2$, which implies that
	$\operatorname{Res}_{P_1}\varphi=\operatorname{Res}_{P_1}\sigma \varphi=\operatorname{Res}_{\sigma(P_1)} \varphi=\operatorname{Res}_{P_2}\varphi.$
	As $\operatorname{Res}_{P_1}\varphi=-\operatorname{Res}_{P_2}\varphi$, we get $\operatorname{Res}_{P_1}\varphi=\operatorname{Res}_{P_2}\varphi=0$, and $\varphi$ is holomorphic at $\{P_1,P_2\}$.
\end{proof}

If for all pairs $\{P_1,P_2\}\subseteq \hat{S}_X$, we have $e_{P_1}=e_{P_2}$, then we can give a positive answer to (\ref{G-differentials}).

\begin{prop}\label{G-invariant}
With the notations above, we have $$\hat{\pi}^* H^{0}(\hat{Y}, \Omega_{\hat{Y}}(\hat{S}_Y))=H^{0}(\hat{X}, \Omega_{\hat{X}}(\pi^{-1}(\hat{S}_Y)))^G.$$
Moreover, if the ramification indexes for any pair $\{P_1,P_2\}\subseteq \hat{S}_X$ are equal, namely $e_{P_1}=e_{P_2}$, then we have the canonical commutative diagram
\begin{equation}\nonumber
	\begin{tikzcd}
		H^{0}(X, \omega_{X})^G \arrow [r,leftarrow,"\pi^*"]\arrow[d,hook] & H^{0}(Y, \omega_{Y}) \arrow[d,hook]\\
		H^{0}(\hat{X}, \Omega_{\hat{X}}(\pi^{-1}(\hat{S}_Y)))^G\arrow [r,leftarrow,"\hat{\pi}^*"] & H^{0}(\hat{Y}, \Omega_{\hat{Y}}(\hat{S}_Y)),
	\end{tikzcd}
\end{equation}
and the rows are both isomorphisms. In particular, we have $\mathrm{dim}_k H^{0}(X, \omega_X)^G=p_a(Y)$.
\end{prop}
\begin{proof}
	Now we treat $\hat{S}_Y$ as a positive divisor here.
	By Proposition \ref{Kani equality} (\ref{G2}), we have 
	$$
	H^{0}(\hat{X}, \Omega_{\hat{X}}(\pi^{-1}(\hat{S}_Y)))^G=\pi^* H^0 (\hat{Y},\Omega_{\hat{Y}}  (\left\lfloor n^{-1} \pi_*(\pi^{-1}(\hat{S}_Y)+R_{\pi}) \right\rfloor)).
	$$
	
	Consider the coefficient of prime divisors in $\left\lfloor n^{-1} \pi_*(\pi^{-1}(\hat{S}_Y)+R_{\pi}) \right\rfloor=\sum a_Q Q$.\\
	(1) If $Q\in Bl(Y)-\hat{S}_Y$, then $a_Q=\lfloor	\frac{e_Q-1}{e_Q} \rfloor=0$;\\
	(2) If $Q\in \hat{S}_Y -Bl(Y)$, then $a_Q=1$;\\
	(3) If $Q\in \hat{S}_Y \cap Bl(Y)$, then  
	$a_Q=\lfloor	\frac{1}{e_Q}+\frac{e_Q-1}{e_Q} \rfloor=1$.
	
	 So we have $\left\lfloor n^{-1} \pi_*(\pi^{-1}(\hat{S}_Y)+R_{\pi}) \right\rfloor=\hat{S}_Y$, hence the isomorphism on the lower row. Note that
	both $H^{0}(X, \omega_{X})^G $ and $H^{0}(Y, \omega_{Y})$ are the subspaces satisfying the residue relations, then we have the isomorphism of the upper row.
\end{proof}

\subsection{Chevalley-Weil formula for irreducible nodal curves}
Let $\chi$ be a $1$-dimensional character of $G$. With the notations in  (\ref{Canonical G-differentials}), we consider the embedding
$$H^0(X,\omega _X)_{\chi}\rightarrow H^{0}(\hat{X}, \Omega_{\hat{X}}(\hat{S}_X))_{\chi}.$$

\begin{prop}\label{chi-set}
Let $Y$ be smooth, set
\begin{equation}
\hat{S}_X^{\chi}=\left\{\hat{P}\in \hat{S}_X \mid \exists \tau\in G_{\alpha(\hat{P})}\text{ s.t. }\tau(\hat{P}) \neq \hat{P} \text{ and }\chi(\tau)=-1 \right\}.
\end{equation}
Then the image of $H^{0}(X, \omega_{X})_{\chi}$ in $H^{0}(\hat{X}, \Omega_{\hat{X}}(\hat{S}_X))_{\chi}$ is equal to $H^{0}(\hat{X}, \Omega_{\hat{X}}(\hat{S}_X^{\chi}))_{\chi}$. So we have an isomorphism
\begin{align}
	H^{0}(X, \omega_{X})_{\chi} \overset{\sim}{\rightarrow} H^{0}(\hat{X}, \Omega_{\hat{X}}(\hat{S}_X^{\chi}))_{\chi}~.
\end{align}
We call $\hat{S}_X^{\chi}$ \emph{the singular $\chi$-set of $X$}.
\end{prop}

\begin{proof}

Assume $\varphi \in  H^{0}(\hat{X}, \Omega_{\hat{X}}(\hat{S}_X^{\chi}))_{\chi}$, and let $\alpha^{-1}(P)=\{P_1,P_2\}\subseteq \hat{S}_X^{\chi}$. So there is a $T\in G_P$ with $T(P_1)=P_2$ and $\chi(T)=-1$ by the definition of $\hat{S}_X^{\chi}$. Hence
$$
-\operatorname{Res}_{P_1}\varphi=\operatorname{Res}_{P_1}T \varphi=\operatorname{Res}_{T^{-1}(P_1)} \varphi=\operatorname{Res}_{P_2}\varphi.
$$
Conversely, given some $\varphi_{0}\in H^{0}(X, \omega_{X})_{\chi}$ with poles on a pair $\alpha^{-1}(P)=\{P_1,P_2\}\subseteq \hat{S}_X$ and some $T\in G_P$ with $T(P_1)=P_2$, which exists by the smooth of $Y$, we have
$$
\chi(T)\operatorname{Res}_{P_1}( \varphi_0) =\operatorname{Res}_{P_1}(T \varphi_0)=\operatorname{Res}_{P_2}( \varphi_0)=-\operatorname{Res}_{P_1}( \varphi_0).
$$
Hence $\chi(T)=-1$ and $\{P_1,P_2\}\subseteq \hat{S}_X^{\chi}$.
\end{proof}

\begin{rmk}
If $\chi=1$ is the trivial representation, then $\hat{S}_X^{\chi}=\varnothing$, which is consistent with \emph{Lemma \ref{G-invariant}} in the case $\hat{S}_Y=\varnothing$.\\
\end{rmk}
\begin{rmk}
	Suppose $\varphi' \in H^{0}(\hat{X}, \Omega_{\hat{X}}(S'))$ has a pole at $P_1 \in S'$, namely $\operatorname{Res}_{P_1}\varphi' \neq 0$. If $\varphi' \in H^{0}(X, \omega_{X})$, then for the pair $\{P_1,P_2\}$, it requires\\
	a) $v_{P_2}(\varphi')=v_{P_1}(\varphi')=-1$; $~$ b) $\operatorname{Res}_{P_2}\varphi'=-\operatorname{Res}_{P_1}\varphi'$.
	
	For a), in general, we can't determine $v_{P_2}(\varphi')$ from the value of $v_{P_1}(\varphi')$. But if $Y$ is smooth, then
	$v_{P_2}(\varphi')=v_{P_1}(\varphi')=v_{P}(\varphi')$ for $\forall P\in \hat{\pi}^{-1}\hat{\pi}(P_1)$.
	
	For b), under the hypothesis of smoothness of $Y$, we use the criterion from the \emph{singular $\chi$-set} to delete these points that can not be poles.
\end{rmk}

Assume $Y=X/G$ is smooth for the rest of this section.\\

Now we compute the dimension of  $H^{0}(\hat{X}, \Omega_{\hat{X}}(\hat{S}_X^{\chi}))_{\chi}$. Let $f_\chi$ be a rational fuction on $\hat{X}$
 such that $\sigma f_\chi=\chi(\sigma) f_\chi,~ \forall \sigma \in G$ and
set $D_{\chi}=\left\lfloor n^{-1} \hat{\pi}_*\left(\hat{S}_X^{\chi}+(f_\chi)+R_{\hat{\pi}}\right)\right\rfloor$.
By Proposition \ref{Kani equality} (\ref{G4}), we have
$$ H^{0}(\hat{X}, \Omega_{\hat{X}}(\hat{S}_X^{\chi}))_{\chi}=f_\chi \cdot \hat{\pi}^* H^{0}(Y,\Omega_Y(D_{\chi})).$$

By Riemann-Roch Theorem, we have 
\begin{align}
\mathrm{dim}_k H^{0}(Y,\Omega_Y(D_{\chi}))
=\mathrm{dim}_k H^{0}(Y,\mathcal{O}_Y(-D_{\chi}))+\operatorname{deg} D_{\chi}+g_Y-1.
\end{align}

\begin{lemma}\label{vanish}
	The space $H^{0}(Y,\mathcal{O}_Y(-D_{\chi}))$ vanishes except when $\chi=1_G$, and in this case, we have $\mathrm{dim}_k H^{0}(Y,-D_{1_G})=1$.
\end{lemma}
\begin{proof}
	We will show that $\operatorname{deg} D_{\chi} > 0$ if $\hat{S}_X^{\chi}\neq \varnothing$ and $D_{\chi}$ is principal if and only if $\chi=1$.	Assume
	$$
	\hat{\pi}_*(f_\chi)=\sum_{Q\in Y} \frac{n}{e_Q}b_Q\cdot Q 
	$$
 where $b_Q=v_P(f_\chi)$, $\forall P\in \hat{\pi}^{-1}(Q)$. Note that
$$
\left\lfloor n^{-1} \hat{\pi}_*((f_\chi)+R_{\hat{\pi}})\right\rfloor=\sum_{Q} \left\lfloor \frac{b_Q+e_Q-1}{e_Q}\right\rfloor Q\geq \sum_{Q}\frac{b_Q}{e_Q}Q,
$$
hence  we have
\begin{equation}
\operatorname{deg}\left\lfloor n^{-1} \hat{\pi}_*\left(\hat{S}_X^{\chi}+(f_\chi)+R_{\hat{\pi}}\right)\right\rfloor \geq \operatorname{deg}\left\lfloor n^{-1} \hat{\pi}_*((f_\chi)+R_{\pi})\right\rfloor\geq n^{-1}\operatorname{deg}(f_\chi)=0.
\end{equation}\\

Wirte 
$
\hat{\pi}_*(\hat{S}_X^{\chi})=\sum_{Q\in Y} \frac{n}{e_Q}c_Q\cdot Q.
$
If $\hat{S}_X^{\chi}\neq \varnothing$, then there is some $c_{Q'}\geq 1$, hence

\begin{align}
\operatorname{deg} D_{\chi}&=\sum_{Q\neq Q'} \left\lfloor \frac{c_Q+b_Q+e_Q-1}{e_Q}\right\rfloor + \left\lfloor \frac{c_{Q'}+b_{Q'}+e_{Q'}-1}{e_{Q'}}\right\rfloor \\\nonumber
&\geq   \sum_{Q\neq Q'} \left\lfloor \frac{c_Q+b_Q+e_Q-1}{e_Q}\right\rfloor + \left\lfloor \frac{b_{Q'}+e_{Q'}}{e_{Q'}}\right\rfloor>\sum_{Q}\frac{b_Q}{e_Q}=0.
\end{align}
Hence $\mathrm{dim}_k H^{0}(Y,\mathcal{O}_Y(-D_{\chi}))=0$ provided $\hat{S}_X^{\chi}\neq \varnothing$.\\

Now we suppose $\hat{S}_X^{\chi}=\varnothing$, and $\operatorname{deg}\left\lfloor n^{-1} \hat{\pi}_*((f_\chi)+R_{\hat{\pi}})\right\rfloor=0$, then we have 
$$\left\lfloor \frac{ b_Q+e_Q-1}{e_Q}\right\rfloor=\frac{b_Q}{e_Q},$$ which implies $b_Q=\lambda_Q e_Q$ for some integer $\lambda_Q$ and 
$D_{\chi}=n^{-1} \hat{\pi}_*(f_\chi)$. If $D_{\chi}$ is principal, namely $\mathrm{dim}_k H^{0}(Y,\mathcal{O}_Y(-D_{\chi}))=1$, then we have $D_{\chi}=(h)$ for some rational function $h\in K(Y)$. Hence $(f_\chi)=(\hat{\pi}^* h)$, which implies $f_\chi \in K(X)^G$, namely $\chi=1$. Conversely, if $\chi=1$, then $f_\chi=\hat{\pi}^* h$ for some rational function $h\in K(Y)$ and $D_{\chi}=(h)$, hence $\mathrm{dim}_k H^{0}(Y,\mathcal{O}_Y(-D_{\chi}))=1$.
\end{proof}

\begin{defn}\label{m(S)}
	Let $S\subseteq \hat{X}$ be a finite subset stable by $G$, and define 
	\begin{equation}
		m_{\chi}(S)=\#\hat{\pi}(S)+\sum_{Q \notin \hat{\pi}(S)}\left\lfloor\frac{e_Q-1}{e_Q}+\frac{1}{n}\left\langle \chi, R_{G, Q}\right\rangle_G\right\rfloor-\frac{1}{n}\left\langle \chi, R_G\right\rangle_G.
	\end{equation}
\end{defn}

\begin{lemma}\label{deg}
We have	$\operatorname{deg} D_{\chi}=m_{\chi}(\hat{S}_X^{\chi})$, which is independent of the choice of $f_{\chi}$.
\end{lemma}
\begin{proof}
	Denote $s_{\chi}=\#\hat{\pi}(\hat{S}_X^{\chi})$.
	Write $n^{-1}\hat{\pi}_*(\hat{S}_X^{\chi})=U^{\chi}+V^{\chi}$, where $\operatorname{Supp}(V^{\chi})=Bl(Y)\cap \hat{\pi}(\hat{S}_X^{\chi})$.
    Suppose $\left(f_{\chi}^n\right)=\hat{\pi}^*(n A+B)$ and $\lfloor n^{-1} B\rfloor=0$ as in lemma \ref{bi}. Write $B=\sum_{Q\in Bl(Y)} b_Q Q$ . 
    By
    \begin{align}
    	\left\lfloor n^{-1} \pi_*\left(\hat{S}_X^{\chi}+(f_\chi)+R_{\hat{\pi}}\right)\right\rfloor
    	=U^{\chi}+A+\left\lfloor  V^{\chi}+n^{-1}B+ n^{-1} \hat{\pi}_*R_{\hat{\pi}}\right\rfloor,
    \end{align}
and $\operatorname{deg}B= -\operatorname{deg}nA$, we have
\begin{align}
	\operatorname{deg} D_{\chi}&=\#U^{\chi}+\sum_{Q \in \hat{\pi}(\hat{S}_X^{\chi})}\left\lfloor1+\frac{b_Q}{n}\right\rfloor+\sum_{Q \notin \hat{\pi}(\hat{S}_X^{\chi})}\left\lfloor\frac{e_Q-1}{e_Q}+\frac{b_Q}{n}\right\rfloor
	-\sum_{Q} \frac{b_Q}{n}\\\nonumber
	&=s_{\chi}+\sum_{Q \notin \hat{\pi}(\hat{S}_X^{\chi})}\left\lfloor\frac{e_Q-1}{e_Q}+\frac{b_Q}{n}\right\rfloor
	-\sum_{Q} \frac{b_Q}{n}.
\end{align}
By Lemma \ref{bi} (\ref{fbi}), we have $b_Q=\left\langle \chi, R_{G, Q}\right\rangle_G$, hence
$$
	\operatorname{deg} D_{\chi}=s_{\chi}+\sum_{Q \notin \hat{\pi}(\hat{S}_X^{\chi})}\left\lfloor\frac{e_Q-1}{e_Q}+\frac{1}{n}\left\langle \chi, R_{G, Q}\right\rangle_G\right\rfloor-\frac{1}{n}\left\langle \chi, R_G\right\rangle_G=m_{\chi}(\hat{S}_X^{\chi}).
$$

\end{proof}

Now we summary above discussion.

For a quotient map $\pi: X \rightarrow X/G=Y$ from an irreducible nodal curve to a smooth curve, we have the induced covering of curves $\hat{\pi}:\hat{X} \rightarrow Y$ with $\hat{S}_X\subseteq \hat{X}$ the preimage of singular locus.
Let $Bl(Y)$ be the branch locus, $R_G$ the ramification module of $\hat{\pi}$, and $R_{G, Q}$ the ramification module of $Q\in Y$.

\begin{thm}[The Chevalley-Weil formula on irreducible nodal curves]\label{irr CW}
	Let $f: X \rightarrow X/G$ be the quotient map of irreducible nodal curves by a finite group $G$ of order $n$.
	Assume $Y=X/G$ is smooth and $\operatorname{char}(k) \nmid n$, then the multiplicity of a given irreducible character $\chi$ is given by 
	\begin{align}\label{irr-cw}
\operatorname{dim}_k  H^{0}(X, \omega_{X})_{\chi}=g_Y-1+ m_{\chi}(\hat{S}_X^{\chi})+\langle \chi, 1_G\rangle.
	\end{align}
	where $\hat{S}_X^{\chi}$ is the \emph{singular $\chi$-set} of $X$ defined in \emph{Proposition \ref{chi-set}}, and 
$$m_{\chi}(\hat{S}_X^{\chi})=	s_{\chi}+\sum_{Q \notin \hat{\pi}(\hat{S}_X^{\chi})}\left\lfloor\frac{e_Q-1}{e_Q}+\frac{1}{n}\left\langle \chi, R_{G, Q}\right\rangle_G\right\rfloor-\frac{1}{n}\left\langle \chi, R_G\right\rangle_G$$
is defined in \emph{Definition \ref{m(S)}}.

In particular, when $\chi=1_G$, we have $\hat{S}_X^{\chi}=\varnothing$ and 
$
\operatorname{dim}_k  H^{0}(X, \omega_{X})^G=g_Y,
$
which is a special case of \emph{Proposition \ref{G-invariant}}.
\end{thm}

\begin{example}[Hyperelliptic stable curves]
	A hyperelliptic stable curve $C$ is a stable curve with a hyperelliptic involution
	$J:C\rightarrow C$, which is an order $2$ automorphism satisfying $C/ \langle J  \rangle=\mathbb{P}^1$. 
	
	Suppose $C$ is an irreducible hyperelliptic stable curve with $N(\geq 1)$ nodes.
	Let $\hat{C}$ be the normalization of $C$ with genus $g$, then $\hat{\pi}:\hat{C}\rightarrow \mathbb{P}^1$ has $2g+2$ fixed(ramification) points, which are all of ramification index $2$, hence $\#Bl(\mathbb{P}^1)=2g+2$.
	
	 There are no branch points in $\hat{\pi}(\hat{S}_C)$, and the Galois group $G=\langle J  \rangle\cong \mathbb{Z}_2$ has two irreducible representations $1_G$ and $\chi^-$ where $\chi^-(J)=-1$.

	For $1_G$, we have 
	$$
	\operatorname{dim}_k  H^{0}(C, \omega_{C})^G=g({\mathbb{P}^1})=0.
	$$
	
	For $\chi^-$, we have $\hat{S}_C^{\chi^-}=\hat{S}_C$, hence $s_{\chi^-}=N$.
	
	And for any fixed point $P$,  the induced character
	$
	\theta_P: G_P= \langle J  \rangle \rightarrow k^*, J\mapsto -1
	$
	is the generator of $\operatorname{Hom}(\mathbb{Z}_2,k^*)$. So we have
	$
	\left\langle \chi^-, R_{G, Q}\right\rangle_G=1$ for all $Q\in Bl(\mathbb{P}^1)$
	and 
	$
	\left\langle \chi^-, R_G\right\rangle_G=2g+2,
	$ which gives
	\begin{align}
			m_{\chi^-}(\hat{S}_C^{\chi^-})&=s_{\chi^-}+\sum_{Q\in Bl(\mathbb{P}^1)}\left\lfloor\frac{e_Q-1}{e_Q}+\frac{1}{2}\left\langle \chi_{-1}, R_{G, Q}\right\rangle_G\right\rfloor-\frac{1}{2}\left\langle \chi^-, R_G\right\rangle_G\\
			&=N+2g+2-(g+1)=p_a(C)+1.
	\end{align}
	Hence $$
		\operatorname{dim}_k  H^{0}(C, \omega_{C})_{\chi^-}=g({\mathbb{P}^1})-1+ m_{\chi^-}(\hat{S}_C^{\chi^-})=p_a(C).
	$$

\end{example}
	
 \section{Nodal curves with several irreducible components}
 In this section, let $X$ be a connected nodal curve and $X=\cup^{d}_{i=1} X_i$ be the decomposition of irreducible components. Consider 
 $$
 \alpha:\coprod X_i\rightarrow X,
 $$
the partial normalization at the intersection locus, then we have the immersion
 \begin{align}
 	H^0(X,\omega_X) \hookrightarrow \oplus_{i=1}^{d} H^{0}(X_i, \omega_{X_i}(I_i)),~\varphi \mapsto (\varphi|_{X_i}),
 \end{align}
 where $I_i$ is the set of intersection points of each $X_i$.
\subsection{Reduction}

Let $G$ be a finite subgroup of $\operatorname{Aut}(X)$ and
 $\pi:X \rightarrow X/G=Y$ be the quotient map. Suppose $Y$ is irreducible, then $G$ acts transitively on $\{X_1,\cdots, X_d\}$ and all these components are isomorphic. Let $G_i$ be the stablizer subgroup of $G$ at $X_i$ and note that the canonical map $X_i/G_i\rightarrow X/G=Y$ is an isomorphism.

\begin{prop}\label{reduction}
	Given a $1$-dimensional character $\chi$ of $G$, we have a commutative diagram:
	\begin{equation}
		\begin{tikzcd}
			H^0(X,\omega_X)_{\chi}\arrow[r, hookrightarrow]\arrow[dr,rightarrow] & \big[\oplus_{i=1}^{d} H^{0}(X_i, \omega_{X_i}(I_i))\big]_{\chi}~,\arrow[d,rightarrow,"p_1"]
			&  \varphi \arrow[r, mapsto]\arrow[dr,mapsto] & (\varphi|_{X_i})_{i=1}^{d} \arrow[d,mapsto]\\
			&H^{0}(X_1, \omega_{X_1}(I_1))_{\chi_1} &  &\varphi|_{X_1}~,
		\end{tikzcd}
	\end{equation}
where $\chi_1$ is the restriction of $\chi$ in $G_1$. Moreover, the projection $p_1$ is an isomorphism.
\end{prop}
\begin{proof}
	Fix some $(\varphi_1,\cdots,\varphi_d)\in \big[\oplus_{i=1}^{d} H^{0}(X_i, \omega_{X_i}(I_i))\big]_{\chi}$. Since $G$ acts on $\{X_1,\cdots, X_d\}$ transitively, then we always have some $T_i:X_i \rightarrow X_1$ and
	$T_i \varphi_1=\chi(T_i)\varphi_i$, hence
	$$
	(\varphi_1,\cdots,\varphi_d)=(~\varphi_1,\chi(T_2)^{-1}T_2\varphi_1,\cdots,\chi(T_d)^{-1}T_d\varphi_1~),
	$$
	namely $(\varphi_1,\cdots,\varphi_d)$ is uniquely determined by $\varphi_1$. Note that $T_1 \varphi_1=\chi(T_1)\varphi_1$ for any $T_1\in G_1$, which implies 
	$\varphi_1 \in H^{0}(X_1, \omega_{X_1}(I_1))_{\chi_1}$, so the projection to the first component $p_1$ is injective.
	
	Conversly, we want to show $(~\varphi_1,\chi(T_2)^{-1}T_2\varphi_1,\cdots,\chi(T_d)^{-1}T_d\varphi_1~)$
	 is the preimgae of $\varphi_1 \in H^{0}(X_1, \omega_{X_1}(I_1))_{\chi_1}$. Note that for any two $\sigma,\tau:X_i \rightarrow X_1$, we have 
	$
	\chi(\sigma)^{-1}\sigma\varphi_1=\chi(\tau)^{-1}\tau\varphi_1.
	$
	Hence for any $T\in G$, we have
	\begin{align}
		T(~\varphi_1,\chi(T_2)^{-1}T_2\varphi_1,\cdots,\chi(T_d)^{-1}T_d\varphi_1~)=\chi(T)(~\varphi_1,\chi(T_2)^{-1}T_2\varphi_1,\cdots,\chi(T_d)^{-1}T_d\varphi_1~).
	\end{align}

\end{proof}

\begin{rmk}\label{irr iso}
%Let $\alpha_i: \hat{X_i}\rightarrow X_i$ be the normalization of each $X_i$.
By the same arguments as in \emph{Proposition \ref{chi-set}} and \emph{Proposition \ref{reduction}}, we
set 
\begin{equation}
	I_i^{\chi}=\left\{P\in I_i\mid \exists \tau\in G_{P}-G_i\text{ s.t. }\chi(\tau)=-1 \right\}
\end{equation}
to be those intersection points that could be the poles of $\varphi|_{X_i}$ for $\varphi\in 	H^0(X,\omega_X)_{\chi}$,
and then we have the isomorphisms
\begin{equation}
	H^0(X,\omega_X)_{\chi}  \overset{\sim}{\rightarrow} [\oplus_{i=1}^{d} H^{0}(X_i, \omega_{X_i}(I_i^{\chi}))]_{\chi}
	\overset{\sim}{\rightarrow} H^{0}(X_1, \omega_{X_1}(I_1^{\chi}))_{\chi_1}.
\end{equation}
\end{rmk}
So our research obeject has been reduced to the irreducible nodal curve acting by the subgroup $G_1$ on $X_1$.

\subsection{Chevalley-Weil formula for connected nodal curves}
With the notations above, note that $I_i^{\chi}=\varnothing$ when $\chi=1_G$. By \emph{Proposition \ref{G-invariant}}, if the ramification indexes of any pair $\{P_1,P_2\}\subseteq \hat{S}_{X_1}$ are equal for $\pi_1:X_1 \rightarrow X_1/G_1=Y$, then we have
\begin{equation}
	H^0(X,\omega_X)^G  \overset{\sim}{\rightarrow} H^{0}(X_1, \omega_{X_1})^{G_1} \overset{\sim}{\rightarrow}H^{0}(Y, \omega_{Y}).
\end{equation}
In this case, we have $\operatorname{dim}_k H^0(X,\omega_X)^G=p_a(Y)$.\\

Assume that $Y$ is smooth for the rest of this section.

Let $\hat{\pi}_1:\hat{X_1} \rightarrow Y$ be the normalization of $\pi_1$, $Bl(Y)$ the branch locus, $R_{G_1}$ the ramification module of $\hat{\pi}_1$, and $R_{{G_1}, Q}$ the ramification module of $Q\in Y$.

Suppose $\hat{S}_{X_1}^{\chi_1}$ is the singular $\chi_1$-set of $X_1$ in \emph{Proposition \ref{chi-set}}, then we have the isomorphisms
\begin{align}
	H^0(X,\omega_X)_{\chi}\overset{\sim}{\rightarrow}H^{0}(X_1, \omega_{X_1}(I_1^{\chi}))_{\chi_1} \overset{\sim}{\rightarrow} H^{0}(\hat{X_1}, \Omega_{\hat{X}_1}(\hat{S}_{X_1}^{\chi_1}\cup I_1^{\chi}))_{\chi_1}.
\end{align}
Note that $n_1:=\# G_1=n/d$ and by \emph{Proposition \ref{Kani equality} (\ref{G4})} again, we have
$$
H^{0}(\hat{X_1}, \Omega_{\hat{X_1}}(\hat{S}_{X_1}^{\chi_1}\cup I_1^{\chi}))_{\chi_1} 
=f_{\chi_1} \cdot \hat{\pi}_1^* H^{0}(Y,\Omega_Y\left\lfloor\ n_1^{-1} \pi_*\left(\hat{S}_{X_1}^{\chi_1}\cup I_1^{\chi}+(f_{\chi_1})+R_{\pi_1}\right)\right\rfloor),
$$
where $f_{\chi_1} \in K(X_1)^*$ is such that $\sigma f_{\chi_1}=\chi_1(\sigma) f_{\chi_1},~ \forall \sigma \in G_1$.
The same argument in Lemma \ref{deg} shows that 
\begin{equation}
	\operatorname{deg} \left\lfloor\ n_1^{-1} \pi_*\left(\hat{S}_{X_1}^{\chi_1}\cup I_1^{\chi}+(f_{\chi_1})+R_{\pi_1}\right)\right\rfloor=m_{\chi_1}(\hat{S}_{X_1}^{\chi_1}\cup I_1^{\chi})
\end{equation}
where by \emph{Definition \ref{m(S)}},
	$$
	m_{\chi_1}(\hat{S}_{X_1}^{\chi_1}\cup I_1^{\chi})=\#\hat{\pi}_1(\hat{S}_{X_1}^{\chi_1}\cup I_1^{\chi})+\sum_{Q \notin \hat{\pi}_1(\hat{S}_{X_1}^{\chi_1}\cup I_1^{\chi})}\left\lfloor\frac{e_Q-1}{e_Q}+\frac{d}{n}\left\langle \chi_1, R_{G_1, Q}\right\rangle_{G_1}\right\rfloor-\frac{d}{n}\langle \chi_1, R_{G_1}\rangle_{G_1}.
$$

By the same argument in \emph{Lemma \ref{vanish}}, and Riemann-Roch Theorem, we have 
\begin{thm}[The Chevalley-Weil formula on connected nodal curves]\label{sereral CW}
	Let $X$ be a connected nodal curve of $d$ irreducible components and $G$ a finite group of order $n$ acting on $X$.
	
	 Assume the quotient curve $Y=X/G$ is smooth(hence irreducible), then we have a canonical map $X_1 /rightarrow =X_1/G_1=Y$, where $X_1$ is an irreducible component and $G_1$ is the  stablizer subgroup of $G$ at $X_1$.
	
	With the notations above, the multiplicity of a $1$-dimensional character $\chi$ is given by
	\begin{equation}
	\operatorname{dim}_k  H^{0}(X, \omega_{X})_{\chi}=g_Y-1+ m_{\chi_1}(\hat{S}_{X_1}^{\chi_1}\cup I_1^{\chi})+ \delta_{\chi},       %\langle \chi^d, 1_{G_1}\rangle,
	\end{equation}
where $\delta_{\chi}=0$ or $1$. And $\delta_{\chi}=1$ if and only if $I_1^{\chi}=\varnothing$ and $\chi_1=1_{G_1}$. In particular, when $\chi=1_G$, we have
	$$
	\operatorname{dim}_k  H^{0}(X, \omega_{X})^G=g_Y.
	$$
\end{thm}

Note that this theorem is exactly the direct generalization of Theorem \ref{irr CW}, since if $d=1$, then $I_1^{\chi}=\varnothing$.

\begin{example}
	Let $C=C_1\cup C_2$ be a hyperelliptic stable curve, where $C_1\approx C_2\approx \mathbb{P}^1$, and they intersect in $\#(C_1\cap C_2)=m(>2)$ points,
	then $p_a(C)=m-1$. 
	Consider the hyperelliptic involution $J$ permuting $C_1$ and $C_2$, we have $\pi: C\rightarrow C/ \langle J \rangle=\mathbb{P}^1$.
	
	The covering map $\pi_1:C_1\rightarrow \mathbb{P}^1$ is an identity.
	For the representation $\chi^-$, we have $\chi^-_1=id$. Hence $\hat{S}_{C_1}^{id}=\varnothing$ and 
	$m_{id}(\hat{S}_{C_1}^{id}\cup I_1^{\chi^-})=\#\pi(I_1^{\chi^-})=m$.
	
	So by \emph{Theorem \ref{sereral CW}}, we have
	\begin{equation}
	\operatorname{dim}_k  H^{0}(C, \omega_{C})_{\chi^-}=g_{\mathbb{P}^1}-1+m=m-1=p_a(C),
	\end{equation}
   and $\operatorname{dim}_k  H^{0}(C, \omega_{C})^G=p_{a}({\mathbb{P}^1})=0$.
\end{example}

\end{document}